\documentclass[11pt,reqno]{amsart}
\usepackage{amsmath, amsthm, xcolor, graphicx, amssymb, mathtools, chngcntr, graphicx,  hyperref}
\usepackage{booktabs}
\usepackage{enumerate}
\usepackage{float}
\usepackage[tmargin=1in,bmargin=1in,rmargin=1in,lmargin=1in]{geometry}

\newcommand{\be}{\begin{equation}}
	\newcommand{\ee}{\end{equation}}
\newcommand{\ben}{\begin{eqnarray*}}
	\newcommand{\een}{\end{eqnarray*}}

\newtheorem{theorem}{Theorem}[section]
\newtheorem{lemma}[theorem]{Lemma}
\newtheorem{corollary}[theorem]{Corollary}
\newtheorem{proposition}[theorem]{Proposition}
\newtheorem{remark}[theorem]{Remark}

\newtheorem{definition}[theorem]{Definition}
\newtheorem{example}[theorem]{Example}

\newcommand{\bt}{\begin{theorem}}
	\newcommand{\et}{\end{theorem}}
\newcommand{\bl}{\begin{lemma}}
	\newcommand{\el}{\end{lemma}}
\newcommand{\bcl}{\begin{corollary}}
	\newcommand{\ecl}{\end{corollary}}
\newcommand{\bex}{\begin{example}}
	\newcommand{\eex}{\end{example}}
\newcommand{\brem}{\begin{remark}}
	\newcommand{\erem}{\end{remark}}
\newcommand{\bed}{\begin{definition}}
	\newcommand{\eed}{\end{definition}}

\numberwithin{equation}{section}
\numberwithin{theorem}{section}

\begin{document}
	\title{SNT-Rank: Kronecker Products and Euclidean Distance Matrices}

	\author{Bharat Pratap Chauhan}
	\address{(B. P. Chauhan) Department of Mathematics, Indian Institute of Technology Gandhinagar, Gujarat 382355, India}
	\email{bharat024pratap@gmail.com; bharat.chauhan@iitgn.ac.in; bc102@snu.edu.in}
	
	\author{Projesh Nath Choudhury}
	\address{(P. N. Choudhury) Department of Mathematics, Indian Institute of Technology Gandhinagar, Gujarat 382355, India}
	\email{projeshnc@iitgn.ac.in}
	
	\date{\today}
	\maketitle
	
	\let\thefootnote\relax

\begin{abstract}
Symmetric nonnegative matrix trifactorizations (SN-Trifactorizations) were introduced by Bukov\v{s}ek--\v{S}migoc [\textit{Linear Algebra Appl.} 2023] as a symmetric analogue of nonnegative matrix factorizations. A SN-Trifactorization of a symmetric nonnegative matrix $A$ is of the form $A = BCB^{T},$ where $B$ and $C$ are nonnegative matrices, with $C$ symmetric. The associated SNT-rank of $A$ is defined as the smallest integer $k$ for which $A$ admits such a factorization with $C \in \mathbb{R}_{+}^{k \times k}$.

In this paper, we derive sharper upper bounds for the SNT-rank of the Euclidean distance matrices considered by Shitov [\textit{Linear Algebra Appl.} 2025] and Bukov\v{s}ek--\v{S}migoc [\textit{Linear Algebra Appl.} 2023]. We also establish several new relationships between the rank and the SNT-rank of  symmetric nonnegative matrices and show that the SNT-rank is submultiplicative with respect to the Kronecker product. Finally, motivated by a conjecture posed in the Dagstuhl Seminar Report 13082, we prove a multiplicativity result for the nonnegative rank under an additional structural assumption. We also partially resolve a conjecture of Vandaele--Gillis--Glineur--Tuyttens [\textit{J. Global Optim.} 2016].
\end{abstract}
    
\noindent 

\section{\textbf{Introduction and main results}}
{\em Throughout this paper, $\mathbb{R}_{+}^{m\times n}$  denotes the set of all \(m\times n\) nonnegative real matrices and  $\mathcal{S}_n^{+}:=\{A\in\mathbb{R}_{+}^{n\times n}:A=A^T\}$, denotes the set of all \(n\times n\) symmetric nonnegative matrices.}
\smallskip

Matrix factorizations involving nonnegative factors play an important role in matrix theory, optimization, combinatorics, statistics, machine learning, and data analysis. Among the most extensively studied examples are nonnegative matrix factorizations and completely positive factorizations, both of which possess rich structural properties and numerous applications.

Given a nonnegative matrix $A\in\mathbb{R}_{+}^{m\times n},$ the nonnegative matrix factorization (NM-Factorization) problem asks to find nonnegative matrices $U\in\mathbb{R}_{+}^{m\times k}$ and $V\in\mathbb{R}_{+}^{n\times k}$ such that $A\approx UV^{T},$ where the approximation is measured with respect to the Frobenius norm. The exact version of the problem,  one asks for the smallest integer \(k\) such that $A=UV^{T}.$
This minimum value of \(k\) is called the nonnegative rank of \(A\) and is denoted by $\operatorname{rk}_{+}(A),$ while $\operatorname{rk}(A)$ denotes the usual rank of $A.$
It is well known that
\[
\operatorname{rk}(A)
\le
\operatorname{rk}_{+}(A)
\le
\min\{m,n\}.
\]
Nonnegative matrix factorization has rich applications in several areas, including text mining, image processing, hyperspectral imaging, computational biology, and clustering. We refer the reader to \cite{ gillis2020siam, lee1999learning} and the references therein for more information.

For symmetric nonnegative matrices, it is natural to consider factorizations that preserve symmetry as well as nonnegativity. A fundamental example is the \emph{completely positive} (CP) factorization. A matrix $A\in\mathcal{S}_{n}^{+}$ is called \emph{completely positive} if
$
A=UU^{T}
$
for some nonnegative matrix $U\in\mathbb{R}_{+}^{n\times k}$. The smallest such integer $k$ is called the \emph{cp-rank} of $A$, denoted by $\operatorname{cp}(A)$; if no such factorization exists, we set $\operatorname{cp}(A)=\infty$. Completely positive matrices and their cp-ranks have been studied extensively in matrix theory and optimization \cite{berman2003completely}.

In this paper, we study the version of symmetric nonnegative matrix trifactorization, which was recently introduced by Bukov\v{s}ek--\v{S}migoc in \cite{bukovvsek2023symmetric}. More precisely, a symmetric nonnegative matrix trifactorization (SN-Trifactorization) of a matrix $A\in\mathcal{S}_{n}^{+}$
is a factorization of the form $A=BCB^{T},$
where $B\in\mathbb{R}_{+}^{n\times k}$ and $C\in\mathcal{S}_{k}^{+}.$ The smallest integer \(k\) for which such a factorization exists is called the SNT-rank of \(A\) and is denoted by $\operatorname{st}_{+}(A).$ For completeness, we set $\operatorname{st}_{+}(0)=0.$ SN-Trifactorization provides a framework for studying low-rank structure in symmetric nonnegative matrices while preserving both nonnegativity and symmetry 
 \cite{arora2013practicalalgorithm,fu2018anchor,ho2008thesis}.
 
 In \cite{bukovvsek2023symmetric}, Bukov\v{s}ek--\v{S}migoc  established several fundamental properties of the SNT-rank, including its relationships with the rank, nonnegative rank, and cp-rank. They later developed a graph-theoretic framework based on set-join covers \cite{bukovvsek2025pattern}, yielding exact formulas for the SNT-rank of several graph classes. Despite these advances, determining the SNT-rank of a symmetric nonnegative matrix remains difficult problem in general. Shitov \cite{shitov2025further} recently investigated the computational complexity of the SNT-rank for symmetric nonnegative matrices with integer entries.

Euclidean distance matrices have attracted considerable attention in the study of the SNT-rank and SN-Trifactorizations. Given ${x}_1,\ldots,x_n\in\mathbb{R}$,
the associated Euclidean distance matrix (EDM) is the $n \times n$ symmetric nonnegative matrix \begin{equation}\label{eqnedm}
A(x_1,\ldots,x_n):=\left(|x_i-x_j|^2\right)_{i,j=1}^n.
\end{equation}

Every EDM has rank three whenever $n\geq 3$ \cite{mmlin2010edm}.  Bukov\v{s}ek and \v{S}migoc \cite{bukovvsek2023symmetric} initiated the study of the SNT-rank of the special family $A(1,\ldots,n)$ and and established the bound \begin{equation}\label{helbnd1}\operatorname{st}_{+}\!\big(A(1,\ldots,n)\big)
 \leq
 \left\lceil \frac{n}{2}\right\rceil +2.\end{equation} Shitov \cite{shitov2025further} subsequently improved this bound to \[\operatorname{st}_{+}\!\big(A(1,\ldots,n)\big)
 \leq
 4\log_2 n +4.\] Our first main result further improves this upper bound.
 
 \begin{theorem}\label{theorem:SNT_rank_Naturalno}
 	$\operatorname{st}_{+}\!\big(A(1,\ldots,n)\big)
 	\leq
 	2\lceil \log_2 n\rceil.$
 \end{theorem}
We also extend the above result to a broader class of  Euclidean distance matrices. A key ingredient in the proof is the behavior of the SNT-rank under the Kronecker product. Recall that the Kronecker product of two matrices $A=(a_{ij})\in\mathbb{R}^{m_1\times n_1}$ and $B\in\mathbb{R}^{m_2\times n_2}$ is defined by
\[
A\otimes B
:=
\begin{pmatrix}
	a_{11}B & \cdots & a_{1n_1}B\\
	\vdots & \ddots & \vdots\\
	a_{m_1 1}B & \cdots & a_{m_1n_1}B
\end{pmatrix}
\in
\mathbb{R}^{m_1m_2\times n_1n_2}.
\]
Our next main result establishes that the SNT-rank is submultiplicative with respect to the Kronecker product.

\begin{theorem}\label{lemma:kroneckerproduct}
	Let $A_1\in\mathcal{S}_{n_1}^{+}$ and $A_2\in\mathcal{S}_{n_2}^{+}$. Then
	\[
	\operatorname{rk}(A_1)\operatorname{rk}(A_2)
	\le
	\operatorname{st}_{+}(A_1\otimes A_2)
	\le
	\operatorname{st}_{+}(A_1)\operatorname{st}_{+}(A_2).
	\]
	Moreover, if $n_1,n_2\le3$, then
	\[
	\operatorname{st}_{+}(A_1\otimes A_2)
	=
	\operatorname{st}_{+}(A_1)\operatorname{st}_{+}(A_2).
	\]
\end{theorem}
The preceding result provides a useful method for constructing symmetric nonnegative matrices of SNT-rank at most 9.

The multiplicative properties of the nonnegative rank with respect to the Kronecker product have been studied extensively. In the Dagstuhl Seminar Report \cite{Dagstuhl report 13082}, it was conjectured that, for any two nonnegative matrices $A_1$ and $A_2$,
\begin{equation}\label{eq1}
	\operatorname{rk}_{+}(A_1\otimes A_2)
	=
	\operatorname{rk}_{+}(A_1)\operatorname{rk}_{+}(A_2).
\end{equation}
Vandaele--Gillis--Glineur--Tuyttens  \cite{vandaele2016heuristics} later disproved this conjecture by constructing matrices for which \eqref{eq1} fails. The following result establishes the multiplicativity identity \eqref{eq1} under an additional assumption on the matrices $A_1$ and $A_2$.

\begin{theorem}\label{theorem:kroneckersntequality_nonnegativerank}
	Let $A_1 \in \mathbb{R}_{+}^{m_1 \times n_1}$ and $A_2 \in \mathbb{R}_{+}^{m_2 \times n_2}$ such that at least one of them has rank $1$. Then
	\[
	\operatorname{rk}_{+}(A_1 \otimes A_2)
	=
	\operatorname{rk}_{+}(A_1)\operatorname{rk}_{+}(A_2).
	\]
\end{theorem}

We now turn our attention to a weaker question. The failure of multiplicativity of the nonnegative rank with respect to the Kronecker product for general nonnegative matrices led Vandaele--Gillis--Glineur--Tuyttens  \cite{vandaele2016heuristics} to propose the following conjecture:
\[
\operatorname{rk}_{+}(A\otimes A)
\ge
\operatorname{rk}_{+}(A) \operatorname{rk}(A) \hbox{ for all } A \in \mathbb{R}_{+}^{m \times n}.
\]
In our final main result, we settle this conjecture for all nonnegative matrices $A \in \mathbb{R}_{+}^{m \times n}$ with either $m\leq 3$ or $n\leq 3$. In fact, we establish the following stronger result.

\begin{theorem}\label{theorem:kroneckersntequality_nonnegativerank2}
Let $A_1 \in \mathbb{R}_{+}^{m_1 \times n_1}$ and $A_2 \in \mathbb{R}_{+}^{m_2 \times n_2}$. Suppose that
either $m_i \le 3$ or $n_i \le 3$ for each $i=1,2$. Then
\[
\operatorname{rk}_{+}(A_1 \otimes A_2)
=
\operatorname{rk}_{+}(A_1)\operatorname{rk}_{+}(A_2).
\]
\end{theorem}

We conclude this introduction by outlining the organization of the paper. In Section~2, we recall the necessary preliminaries and establish several basic properties of the SNT-rank of symmetric nonnegative matrices. Section~3 is devoted to studying the behavior of the SNT-rank under the Kronecker product, culminating in the proofs of Theorems~\ref{lemma:kroneckerproduct}, \ref{theorem:kroneckersntequality_nonnegativerank}, and \ref{theorem:kroneckersntequality_nonnegativerank2}. In the final section, we investigate the SNT-rank of Euclidean distance matrices and prove Theorem~\ref{theorem:SNT_rank_Naturalno}.

\section{\textbf{Basic properties of SNT-rank}}
In this section, we develop several fundamental properties of the SNT-rank that play a key role in the proofs of our main results. We begin by recalling some basic definitions and preliminary results that will be used throughout the paper. The following preliminary result describes the relationship among the classical rank, the SNT-rank, the nonnegative rank, and the CP-rank of a symmetric nonnegative matrix.

\begin{proposition}\cite[Proposition~2.1]{bukovvsek2023symmetric}
	\label{prel:proposition2.1}
	Let $A\in\mathcal{S}_n^{+}.$
	Then:
	\begin{enumerate}[(a)]
		\item
		$\operatorname{rk}(A)
		\le
		\operatorname{st}_{+}(A)
		\le n.$
		
		\item
		$\operatorname{rk}_{+}(A)
		\le
		\operatorname{st}_{+}(A)
		\le
		\operatorname{cp}(A).$
	\end{enumerate}
\end{proposition}
The following two results  provides some interesting properties of SNT-rank.

\begin{proposition}\cite[Proposition~2.2]{bukovvsek2023symmetric}
	\label{prel:proposition2.2}
	Let $A,A'\in\mathcal{S}_n^{+},$ $A''\in\mathcal{S}_m^{+},$
	and let \(A_0\) be a principal submatrix of \(A\). Then:
	\begin{enumerate}[(a)]
		\item
		$\operatorname{st}_{+}(A+A')
		\le
		\operatorname{st}_{+}(A)
		+
		\operatorname{st}_{+}(A').$
		
		\item
		$\operatorname{st}_{+}(A_0)
		\le
		\operatorname{st}_{+}(A).$
		
		\item
		$
		\operatorname{st}_{+}(A'\oplus A'')
		=
		\operatorname{st}_{+}(A')
		+
		\operatorname{st}_{+}(A'').
		$
	\end{enumerate}
\end{proposition}

\begin{proposition}\cite[Corollary~2.1]{bukovvsek2023symmetric}
	\label{prel:proposition2.3}
	Let $A\in\mathcal{S}_n^{+}$
be a matrix of rank \(2\). Then $\operatorname{st}_{+}(A)=2.$
\end{proposition}
The next preliminary result provides a lower bound for the nonnegative rank of a matrix \(A\in\mathbb{R}_{+}^{m\times n}\). To state this result, we first recall a basic definition.

\begin{definition}\cite{cohen1993nonnegative}
	\label{prel:definition_independent}
	Let $A\in\mathbb{R}_{+}^{m\times n}$
	be a nonnegative matrix. Two entries \(a_{ij}\) and \(a_{kp}\) of \(A\) are called independent if
	\[
	a_{ij}a_{kp}>0
	\qquad\text{and}\qquad
	a_{ip}a_{kj}=0.
	\]
\end{definition}

\begin{proposition}\cite[Lemma~2.4]{cohen1993nonnegative}
	\label{prel:theorem_independent}
	Let $A\in\mathbb{R}_{+}^{m\times n}$ be a nonnegative matrix. If \(A\) contains a set of \(q\) pairwise independent entries, then $\operatorname{rk}_{+}(A)\ge q.$
\end{proposition}
\smallskip

We now construct a family of symmetric nonnegative matrices for which all the inequalities in Proposition~\ref{prel:proposition2.1} are strict. The construction is based on combining matrices that exhibit different gaps among the classical rank, nonnegative rank, SNT-rank, and CP-rank.

\begin{example}\label{ex:strict_inequalities}
Let $A=A_1\oplus A_2,$ where
\[
A_1=
\begin{pmatrix}
a & b & 0 & 0\\
b & 0 & c & 0\\
0 & c & 0 & d\\
0 & 0 & d & e
\end{pmatrix}
\quad
\text{and}
\quad
A_2=
\begin{pmatrix}
1 & 1 & 2 & 2\\
1 & 0 & 1 & 2\\
2 & 1 & 0 & 1\\
2 & 2 & 1 & 1
\end{pmatrix},
\]
with \(a,b,c,d,e>0\) satisfying $b^{2}d^{2}=ac^{2}e.$ Note that $\operatorname{rk}(A_1)=3$ and the entries
\[
(A_1)_{11},\quad (A_1)_{23},\quad (A_1)_{32},
\quad \text{and} \quad
(A_1)_{44}
\]
are pairwise independent in the sense of Definition~\ref{prel:definition_independent}. By Proposition~\ref{prel:theorem_independent}, $\operatorname{rk}_{+}(A_1)\geq4.$  Since $A_1 \in \mathbb{R}^{4 \times 4}_{+}$, Proposition~\ref{prel:proposition2.1} yields $\operatorname{st}_{+}(A_1)=\operatorname{rk}_{+}(A_1)=4.$ Next, one can verify that $\operatorname{rk}(A_2)=3.$ Moreover, by \cite[Example 2.3]{bukovvsek2023symmetric}, $\operatorname{rk}_{+}(A_2)=3$ and $\operatorname{st}_{+}(A_2)=4.$ Since \(A=A_1\oplus A_2\), $\operatorname{rk}(A)= 6$, $\operatorname{rk}_{+}(A) = 7$, and by Proposition~\ref{prel:proposition2.2}(c), $\operatorname{st}_{+}(A) = 8.$ Therefore,
\[
\operatorname{rk}(A)
<
\operatorname{rk}_{+}(A)
<
\operatorname{st}_{+}(A).
\]
Finally, the fact that \(A_2\) is not positive semidefinite implies that \(\operatorname{cp}(A)=\infty\). Therefore, every inequality in Proposition~\ref{prel:proposition2.1} is strict for the matrix \(A\).
\end{example}

\medskip

We next investigate the SNT-rank of symmetric nonnegative matrices of rank at most three. 

\begin{proposition}\label{prop:rank-snt-small}
Let $A \in \mathcal{S}_n^{+}$. Then the following statements hold:
\begin{enumerate}[(a)]
    \item For each $k \in \{1,2\}$, $\operatorname{rk}(A)=k$ if and only if $\operatorname{st}_{+}(A)=k.$

    \item If $\operatorname{st}_{+}(A)=3$, then $\operatorname{rk}(A)=3.$
\end{enumerate}
\end{proposition}

\begin{proof}
$(a)$ We first assume that that $\operatorname{rk}(A)=1$. Then there exists a nonzero vector $v \in \mathbb{R}_{+}^{n}$ such that $A = vv^{\mathsf T}.$ Hence, $A$ admits an SN-Trifactorization with one column, and therefore $\operatorname{st}_{+}(A)\leq 1.$ Since $\operatorname{rk}(A)\leq \operatorname{st}_{+}(A),$ $\operatorname{st}_{+}(A)=1.$

Conversely, suppose that $\operatorname{st}_{+}(A)=1$. Then $A\neq 0$ and by Proposition~\ref{prel:proposition2.1}, $\operatorname{rk}(A)=1.$
\medskip

If $\operatorname{rk}(A)=2$, by Proposition~\ref{prel:proposition2.3}, $\operatorname{st}_{+}(A)=2.$ To prove the converse, assume that $\operatorname{st}_{+}(A)=2$. By Proposition~\ref{prel:proposition2.1}$(a)$, $\operatorname{rk}(A)\leq 2.$ Since $\operatorname{rk}(A)=1$ if and only if $\operatorname{st}_{+}(A)=1$, it follows that $\operatorname{rk}(A)=2.$

\medskip

$(b)$ Suppose that $\operatorname{st}_{+}(A)=3$. Using Proposition~\ref{prel:proposition2.1}$(a)$, we have $\operatorname{rk}(A)\leq 3.$
If $\operatorname{rk}(A)\leq 2$, then part~$(a)$ would imply that $\operatorname{st}_{+}(A)\leq 2,$
which contradicts the assumption. Therefore, $\operatorname{rk}(A)=3.$
\end{proof}


As an immediate consequence, the rank and the SNT-rank coincide for symmetric nonnegative matrices of order at most three.

\begin{corollary}\label{cor:rk=snt<=3}
Let $A \in \mathcal{S}_n^{+}$, where $ n\leq 3$. Then $\operatorname{rk}(A)=\operatorname{st}_{+}(A).$
\end{corollary}

\begin{proof}
By Proposition \ref{prel:proposition2.1},  $\operatorname{rk}(A)\leq \operatorname{st}_{+}(A) \leq n.$ Since $\operatorname{rk}(A)\leq 3$, the conclusion therefore follows directly from Proposition~\ref{prop:rank-snt-small}.
\end{proof}
The following example shows that the conclusion of Corollary \ref{cor:rk=snt<=3} does not hold in general.
\begin{example}\label{example:LEDM6_st=5}
Consider the matrix
\[
A=
\begin{pmatrix}
0 & 1 & 4 & 9 & 16 & 25\\
1 & 0 & 1 & 4 & 9 & 16\\
4 & 1 & 0 & 1 & 4 & 9 \\
9 & 4 & 1 & 0 & 1 & 4\\
16 & 9 & 4 & 1 & 0 & 1\\
25 & 16 & 9 & 4 & 1 & 0
\end{pmatrix}.
\]
Then $\operatorname{rk}(A)=3$ and by \cite[Example 2]{gillis2012geometric}, $\operatorname{rk}_{+}(A)=5.$ Observe that $A$ has an SN-Trifactorization $A=BCB^{T}$ with
\[
B=
\begin{pmatrix}
5 & 0 & 1 & 0 & 0\\
3 & 0 & 0 & 1 & 0\\
1 & 0 & 0 & 0 & 1\\
0 & 1 & 0 & 0 & 1\\
0 & 3 & 0 & 1 & 0\\
0 & 5 & 1 & 0 & 0
\end{pmatrix}
\qquad
\text{and}\qquad C=
\begin{pmatrix}
0 & 1 & 0 & 0 & 0\\
1 & 0 & 0 & 0 & 0\\
0 & 0 & 0 & 1 & 4\\
0 & 0 & 1 & 0 & 1\\
0 & 0 & 4 & 1 & 0
\end{pmatrix}.
\]
Thus, $\operatorname{st}_{+}(A)\leq 5.$ Since $\operatorname{rk}_{+}(A)\leq \operatorname{st}_{+}(A),$ $\operatorname{st}_{+}(A)=5.$ Thus, $\operatorname{st}_{+}(A)> \operatorname{rk}(A)$ showing that the equality in Corollary~\ref{cor:rk=snt<=3} does not hold in general.
\end{example}
%
%
Proposition~\ref{prel:proposition2.1} establishes that, for every symmetric nonnegative matrix $A$ of order $n$, $\operatorname{st}_{+}(A)\le n.$
This naturally raises the question of which values of $\operatorname{st}_{+}(A)$ can occur over the class $\mathcal{S}_n^{+}$, and of the structural properties of matrices attaining each such value. The following result partially answers this question.



\begin{proposition}\label{Basic results:proposition1}
For every integer $n \geq 1$ and each $k$ with $1 \leq k \leq n$, there exists a matrix $A \in \mathcal{S}_n^{+}$ such that $\operatorname{st}_{+}(A) = k$.	
%
\end{proposition}

\begin{proof} 
%
 Fix $n \geq 1$ and $1 \leq k \leq n$. Let $A_{1} = vv^{T}$, where $0\neq v \in \mathbb{R}_+^{n-k+1}$ and $A_{2}= I_{k-1}$ be the identity matrix of order $k-1$. Now, we define
\[
A := A_{1} \oplus A_{2} = (vv^{T}) \oplus I_{k-1}.
\]
By Proposition~\ref{prel:proposition2.2}(c), $\operatorname{st}_{+}(A)= \operatorname{st}_{+}(A_1) +\operatorname{st}_{+}(A_2)= k$.
\end{proof}
 

In \cite[Proposition~2.3]{bukovvsek2023symmetric}, the authors established a relationship between \(\operatorname{st}_{+}(A+A^{T})\) and \(\operatorname{rk}_{+}(A)\) for every matrix \(A\in\mathbb{R}_{+}^{n\times n}\). We next derive a similar relationship between the SNT-rank of a matrix of the form \(AA^T\) and the nonnegative rank of the underlying rectangular matrix \(A\).

\begin{proposition}\label{Basic Results: lemma1}
Let $A\in\mathbb{R}_{+}^{m\times n}.$
Then
\[
\operatorname{rk}(A)
\le
\operatorname{st}_{+}(AA^{T})
\le
\operatorname{rk}_{+}(A).
\]
\end{proposition}

\begin{proof}
Let $A= UV^T\in\mathbb R_{+}^{m\times n}$ be a nonnegative matrix factorization of $A$ attaining $\operatorname{rk}_{+}(A)$, where $U\in\mathbb R_{+}^{m\times \operatorname{rk}_{+}(A)}$ and $V\in\mathbb R_{+}^{n\times \operatorname{rk}_{+}(A)}.$ Then
\[
AA^{T}
=
UV^{T}(UV^{T})^{T}
=
U(V^{T}V)U^{T}.
\]
is an SN-Trifactorization of \(AA^{T}\) with $V^{T}V \in \mathcal{S}_{\operatorname{rk}_{+}(A)}^{+}$. Thus,
\[
\operatorname{st}_{+}(AA^{T})
\le
\operatorname{rk}_{+}(A).
\]
Since $\operatorname{rk}(A)
=
\operatorname{rk}(AA^{T})$, by Proposition~\ref{prel:proposition2.1},
$
\operatorname{rk}(A)
\le
\operatorname{st}_{+}(AA^{T})
\le
\operatorname{rk}_{+}(A).
$
\end{proof}

The above proposition provides a useful tool for determining the SNT-rank of symmetric nonnegative matrices. We illustrate this with following  examples.

\begin{example}\label{example:rectangle_snt_rk}
Consider the matrices
\[
M=
\begin{pmatrix}
153 & 8 & 32+11a & 205 & 34\\
8 & 8 & 8 & 10 & 6\\
32+11a & 8 & 32+a^2 & 40+15a & 12+2a\\
205 & 10 & 40+15a & 275 & 45\\
34 & 6 & 12+2a & 45 & 10
\end{pmatrix}
\qquad
\text{and}
\qquad
A=
\begin{pmatrix}
4 & 0 & 4 & 11\\
0 & 2 & 2 & 0\\
4 & 0 & 4 & a\\
5 & 0 & 5 & 15\\
1 & 1 & 2 & 2
\end{pmatrix},
\]
where \(a>0\). Then $M=AA^T$ and $\operatorname{rk}(A)=3.$ Moreover, $A=UV^{T},$ where
\[
U=
\begin{pmatrix}
4 & 0 & 11\\
0 & 2 & 0\\
4 & 0 & a\\
5 & 0 & 15\\
1 & 1 & 2
\end{pmatrix}
\quad
\text{and}
\quad
V=
\begin{pmatrix}
1 & 0 & 0\\
0 & 1 & 0\\
1 & 1 & 0\\
0 & 0 & 1
\end{pmatrix}.
\]
Thus $\operatorname{rk}_{+}(A)\le3$, and by Proposition \ref{prel:proposition2.1}, $\operatorname{rk}_{+}(A)= 3.$ Using Proposition~\ref{Basic Results: lemma1}, we have $\operatorname{st}_{+}(M)=3.$
\end{example}

\begin{remark}
Consider the matrix
\[A=\begin{pmatrix}
	1 & 1 & 2 & 2\\
	1 & 0 & 1 & 2\\
	2 & 1 & 0 & 1\\
	2 & 2 & 1 & 1
\end{pmatrix}.
\]
 In \cite[Example~2.4]{bukovvsek2023symmetric}, it is shown that
$\operatorname{st}_{+}(A^{2})=3$ by constructing an SN-Trifactorization.
We observe that this conclusion also follows directly from
Proposition~\ref{Basic Results: lemma1}. Since $\operatorname{rk}(A)=3$, and by~\cite[Example~2.3]{bukovvsek2023symmetric}, $\operatorname{rk}_{+}(A)=3$, Proposition~\ref{Basic Results: lemma1} yields $\operatorname{st}_{+}(A^{2})=3.$
\end{remark}

In \cite[Lemma~2.1]{bukovvsek2023symmetric}, it was shown that permutation and positive diagonal congruence preserve the existence of an SN-Trifactorization of a given size for symmetric nonnegative matrices. As an immediate consequence, we show that the SNT-rank is invariant under permutation and positive diagonal congruence. This invariance will play a crucial role in the proof of our first main result.


\begin{proposition}\label{proposition:principal_rearrangement/diagonal_scaling}
Let $A \in \mathcal{S}_n^{+}$. Then the following holds.
\begin{enumerate}[(a)]
	\item $\operatorname{st}_{+}(A)=\operatorname{st}_{+}(PAP^{T})$ for every permutation matrix $P \in \mathbb{R}^{n \times n}$.
	
	\item $\operatorname{st}_{+}(A)=\operatorname{st}_{+}(DAD^{T})$ for every diagonal matrix $D \in \mathbb{R}^{n \times n}$ with positive diagonal entries.
\end{enumerate}

\end{proposition}

\begin{proof}
$(a)$ Let $A = B C B^T$ be an SN-Trifactorization of $A$ that achieves $\operatorname{st}_{+}(A)$. Let $P \in \mathbb{R}^{n \times n}$ be a permutation matrix. Then $P A P^T = (P B) C (P B)^T.$ Since $P B \in \mathbb{R}_{+}^{n \times \operatorname{st}_{+}(A)}$, it follows that
\[
\operatorname{st}_{+}(P A P^T) \le \operatorname{st}_{+}(A).
\]
Since $A = P^T (P A P^T) P,$ by the previous steps, $
\operatorname{st}_{+}(A) \le \operatorname{st}_{+}(P A P^T).$
Thus
\[
\operatorname{st}_{+}(A) = \operatorname{st}_{+}(P A P^T).
\]

$(b)$ The proof follows by an argument analogous to that of part $(a)$, together with the fact that $D, D^{-1} \ge 0$.
\end{proof}

Proposition~\ref{proposition:principal_rearrangement/diagonal_scaling} motivates the following result, which will be useful in analyzing the behavior of the SNT-rank under nonnegative matrix congruence transformations.

\begin{lemma}\label{lemma:XTAX construction}
Let $A \in \mathcal{S}_n^{+}$ and let $X \in \mathbb{R}_{+}^{n \times n}$ be a nonnegative matrix. Then
\[
\operatorname{st}_{+}(X A X^T) \leq \operatorname{st}_{+}(A).
\]
\end{lemma}

\begin{proof}
Let $A = B C B^{T}$ be an SN-Trifactorization of $A$ that achieves $\operatorname{st}_{+}(A)$. Then
\[
X A X^T = XB C B^{T} X^T = (X B)\, C\, (X B)^{T}.
\]
Thus, the above representation is an SN-Trifactorization of $X A X^{T}$ with $X B \in \mathbb{R}_{+}^{n \times \operatorname{st}_{+}(A)}$. Hence
\[
\operatorname{st}_{+}(X A X^T) \leq \operatorname{st}_{+}(A).
\]
\end{proof}

Note that the inequality in Lemma~\ref{lemma:XTAX construction} can be strict.

\begin{example}
Consider the matrices
\[
A = \begin{pmatrix}
1 & 0\\
0 & 1
\end{pmatrix} \quad \text{and} \quad
X =
\begin{pmatrix}
1 & 1\\
1 & 1
\end{pmatrix}.
\]
Since $\operatorname{rk}(A) = 2$, by Proposition~\ref{prop:rank-snt-small}, $\operatorname{st}_{+}(A) = 2$. However,
\[
X A X^{T} = X X^{T} = \begin{pmatrix}
2 & 2\\
2 & 2
\end{pmatrix} = \begin{pmatrix}
1\\
1 
\end{pmatrix} \begin{pmatrix}
2\\
\end{pmatrix}\begin{pmatrix}
1 & 1\\ 
\end{pmatrix}.
\]
Thus $\operatorname{st}_{+}(XAX^{T}) < \operatorname{st}_{+}(A),$ showing that the inequality in Lemma~\ref{lemma:XTAX construction} can be strict.
\end{example}

In Lemma~\ref{lemma:XTAX construction}, we established that for a symmetric nonnegative matrix $A \in \mathcal{S}_n^{+}$ and a nonnegative matrix $X \in \mathbb{R}_+^{n \times n}$, $
\operatorname{st}_{+}(X A X^T) \leq \operatorname{st}_{+}(A).
$ This naturally leads to the following question: for which matrices $A \in \mathcal{S}_n^{+}$ does the equality
\[
\operatorname{st}_{+}(X A X^T) = \operatorname{st}_{+}(A)
\]
hold for a given nonnegative matrix $X \in \mathbb{R}_+^{n \times n}$? A characterization of such matrices would provide further insight into the behavior of the SNT-rank under congruence transformations. The following result provides a partial answer to this question.

\begin{proposition}\label{pro:rk_1/2:XAX construction}
Let $A \in \mathcal{S}_n^{+}$ and let $X \in \mathbb{R}^{n \times n}_{+}$  such that $XAX^T\neq 0$. Then
\begin{enumerate}[(a)]
    \item If $\operatorname{rk}(A) = 1$, then 
    $\operatorname{st}_{+}(XAX^{T}) = \operatorname{st}_{+}(A).$
    
    \item If $\operatorname{rk}(A) = 2$, then $\operatorname{st}_{+}(XAX^{T}) = \operatorname{st}_{+}(A)$ if and only if $\operatorname{rk}(XAX^{T}) = 2.$
\end{enumerate}
\end{proposition}

\begin{proof}
$(a)$ Let $A \in \mathcal{S}_n^{+}$ with $\operatorname{rk}(A) = 1$. Then there exists a nonzero vector $u \in \mathbb{R}_{+}^{n}$ such that $A = uu^{T}$.  For any $X \in \mathbb{R}^{n \times n}_{+}$,
\[
XAX^{T} = Xuu^{T}X^{T} = (Xu)(Xu)^{T}.
\]
Since $XAX^T\neq 0$,  $XAX^{T}$ is a rank-one nonnegative symmetric matrix, and by Proposition \ref{prop:rank-snt-small},
\[
\operatorname{st}_{+}(XAX^{T})=1=\operatorname{st}_{+}(A).
\]

\medskip
\noindent
$(b)$ Suppose that $\operatorname{rk}(A)=2$. By Proposition~\ref{prop:rank-snt-small}, we have $\operatorname{st}_{+}(A)=2$. First, assume that $\operatorname{rk}(XAX^{T})=2$. Since $\operatorname{rk}(XAX^{T}) \le \operatorname{st}_{+}(XAX^{T})$, it follows that
\[
2 \le \operatorname{st}_{+}(XAX^{T}) \le 2=\operatorname{st}_{+}(A).
\]
Thus $\operatorname{st}_{+}(XAX^{T})=\operatorname{st}_{+}(A)$.

Conversely, suppose that
$
\operatorname{st}_{+}(XAX^{T})=\operatorname{st}_{+}(A)=2.
$
Since $\operatorname{rk}(XAX^{T}) \le \operatorname{st}_{+}(XAX^{T})$, we have $\operatorname{rk}(XAX^{T}) \le 2$. Thus $\operatorname{rk}(XAX^{T}) \in \{1,2\}$. If $\operatorname{rk}(XAX^{T})=1$, then by Proposition~\ref{prop:rank-snt-small}, we obtain $\operatorname{st}_{+}(XAX^{T})=1$, which contradicts the assumption. Therefore, $\operatorname{rk}(XAX^{T})=2.$
\end{proof}

\begin{example}\label{ex:rank3_counterexample}
Proposition~\ref{pro:rk_1/2:XAX construction} establishes that for a symmetric nonnegative matrix $A \in \mathcal{S}_n^{+}$ and a nonnegative matrix $X \in \mathbb{R}_+^{n \times n}$, the equality $\operatorname{st}_{+}(X A X^T) = \operatorname{st}_{+}(A)$ holds whenever $\operatorname{rk}(XAX^{T})=\operatorname{rk}(A) \leq 2.$ The following example shows that this conclusion fails in general when $\operatorname{rk}(A) \geq 3$, even if $\operatorname{rk}(XAX^{T})=\operatorname{rk}(A)$. Consider the matrix
\[
A =
\begin{pmatrix}
1 & 1 & 0 & 0\\
1 & 0 & 1 & 0\\
0 & 1 & 0 & 1\\
0 & 0 & 1 & 1
\end{pmatrix}
\qquad \text{and} \qquad
X =
\begin{pmatrix}
1 & 1 & 0 & 0\\
0 & 1 & 1 & 0\\
0 & 0 & 1 & 1\\
0 & 0 & 0 & 1
\end{pmatrix}.
\]
Then $\operatorname{rk}(A)=\operatorname{rk}(XAX^T)=3$, and by \cite[Example 2.1]{bukovvsek2023symmetric}, $\operatorname{st}_{+}(A)=4.$ Note that $XAX^T$ admits the SN-Trifactorization
\[
XAX^T =
\begin{pmatrix}
1 & 0 & 0\\
\frac{1}{2} & \frac{1}{2} & 1\\
0 & 1 & 2\\
0 & 1 & 0
\end{pmatrix}
\begin{pmatrix}
3 & 0 & \frac{1}{2}\\
0 & 1 & \frac{1}{2}\\
\frac{1}{2} & \frac{1}{2} & 0
\end{pmatrix}
\begin{pmatrix}
1 & \frac{1}{2} & 0 & 0\\
0 & \frac{1}{2} & 1 & 1\\
0 & 1 & 2 & 0
\end{pmatrix},
\]
which implies $\operatorname{st}_{+}(XAX^T) \leq 3$. By Proposition \ref{prel:proposition2.1}, we have $\operatorname{st}_{+}(XAX^T) = 3.$ Hence, $\operatorname{st}_{+}(XAX^T) < \operatorname{st}_{+}(A)$.

\end{example}



\section{\textbf{SNT-rank and nonnegative rank with respect to the Kronecker product}}

In this section, we investigate the behavior of the SNT-rank of symmetric nonnegative matrices under the Kronecker product and prove our main results -- Theorems~\ref{lemma:kroneckerproduct}, \ref{theorem:kroneckersntequality_nonnegativerank}, and \ref{theorem:kroneckersntequality_nonnegativerank2}. In particular, we establish its submultiplicativity and derive several corresponding bounds. We also obtain results related to conjectures concerning the nonnegative rank under the Kronecker product.

\subsection{\textbf{Kronecker Product and the SNT-Rank}}

We begin by  proving Theorem \ref{lemma:kroneckerproduct} which establishes the submultiplicativity of the SNT-rank with respect to the Kronecker product.


\begin{proof}[Proof of Theorem \ref{lemma:kroneckerproduct}]
Let $A_1 = B_1 C_1 B_1^{T}$ and $A_2 = B_2 C_2 B_2^{T}$ be SN-Trifactorization of $A_1$ and $A_2$ that achieve $\operatorname{st}_{+}(A_1)$ and $\operatorname{st}_{+}(A_2)$ respectively, where $B_i \in \mathbb{R}_{+}^{n_i \times \operatorname{st}_{+}(A_i)}$ and $C_i \in \mathcal{S}_{\operatorname{st}_{+}(A_i)}^{+}$ for $i=1,2$. Using basic properties of the Kronecker product, we obtain
\begin{align*}\label{equation:SNTkronecker}
A_1 \otimes A_2 
&= (B_1C_1B_1^{T}) \otimes (B_2C_2B_2^{T})\nonumber \\[6pt]
&= (B_1\otimes B_2)(C_1\otimes C_2)(B_1 \otimes B_2)^{T}.
\end{align*}
Since $B_1 \otimes B_2 \ge 0$ and $C_1 \otimes C_2 \in \mathcal{S}_{\operatorname{st}_{+}(A_1) \operatorname{st}_{+}(A_2)}^{+}$, the above expression gives an SN-Trifactorization of $A_1 \otimes A_2$. Hence,
$$
\operatorname{st}_{+}(A_1 \otimes A_2) 
\leq 
\operatorname{st}_{+}(A_1) \operatorname{st}_{+}(A_2).
$$ 
The lower bound follows directly from the identity $\operatorname{rk}(A_1 \otimes A_2) = \operatorname{rk}(A_1)\operatorname{rk}(A_2)$ together with Proposition~\ref{prel:proposition2.1}.

For the second assertion, assume that $n_1,n_2\le3$ and $A_{1} \in \mathcal{S}_{n_1}^{+},~A_{2} \in \mathcal{S}_{n_2}^{+}$. By  Corollary \ref{cor:rk=snt<=3}, $\operatorname{rk}(A_1)
=
\operatorname{st}_{+}(A_1)$ and $\operatorname{rk}(A_2)
=
\operatorname{st}_{+}(A_2)$. Using the first part of the proof, we have
\[
\operatorname{st}_{+}(A_1 \otimes A_2)
=
\operatorname{st}_{+}(A_1)\operatorname{st}_{+}(A_2).
\]
\end{proof}

Theorem \ref{lemma:kroneckerproduct} provides a useful method for constructing symmetric nonnegative matrices with SNT-rank at most 9.

\begin{example}
Consider the matrices
\[
A_{1}=
\begin{pmatrix}
1 & 1 & 0\\
1 & 1 & 0\\
0 & 0 & 1
\end{pmatrix}
\qquad\text{and}\qquad
A_{2}=
\begin{pmatrix}
1 & 0 & 1\\
0 & 1 & 1\\
1 & 1 & 2
\end{pmatrix}.
\]
Note that $\operatorname{rk}(A_{1})
=
\operatorname{rk}(A_{2})
=
2$ and, by Corollary~\ref{cor:rk=snt<=3}, $\operatorname{st}_{+}(A_{1})
=
\operatorname{st}_{+}(A_{2})
=
2.$ Then $A_{1}\otimes A_{2} \in \mathcal{S}_{9}^{+}$ and by Theorem \ref{lemma:kroneckerproduct}, $\operatorname{st}_{+}(A_{1}\otimes A_{2})
=
4.$
\end{example}

\begin{remark}\label{remmult}
Let $A_{1} \in \mathcal{S}_{n_1}^{+}$ and $A_{2} \in \mathcal{S}_{n_2}^{+}$ with $\operatorname{st}_{+}(A_1)=\operatorname{rk}(A_1)$ and  $\operatorname{st}_{+}(A_2)=\operatorname{rk}(A_2)$ . Then, as an immediate consequence of Theorem~\ref{lemma:kroneckerproduct}, it follows that  \[\operatorname{st}_{+}(A_1 \otimes A_2)
=
\operatorname{st}_{+}(A_1) \operatorname{st}_{+}(A_2).\]  The following example illustrates this consequence.
\end{remark}
\begin{example}
	Consider the matrices
	\[
	A_{1}=
	\begin{pmatrix}
		1 & 0 & 1\\
		0 & 1 & 1\\
		1 & 1 & 2
	\end{pmatrix}
	\qquad\text{and}\qquad
	A_{2}=
	\begin{pmatrix}
		153 & 8 & 32+11a & 205 & 34\\
		8 & 8 & 8 & 10 & 6\\
		32+11a & 8 & 32+a^2 & 40+15a & 12+2a\\
		205 & 10 & 40+15a & 275 & 45\\
		34 & 6 & 12+2a & 45 & 10
	\end{pmatrix},
	\]
	where \(a>0\). Since $\operatorname{rk}(A_{1})=2,$ it follows from Corollary~\ref{cor:rk=snt<=3} that $\operatorname{st}_{+}(A_{1})=2.$ Moreover, for every \(a>0\), $\operatorname{rk}(A_{2})=3,$ and $\operatorname{st}_{+}(A_{2})=3$ (see the matrix $M$ in Example~\ref{example:rectangle_snt_rk}). By Remark~\ref{remmult},
	\[
	\operatorname{st}_{+}(A_1 \otimes A_2)
	=
	\operatorname{st}_{+}(A_1)\operatorname{st}_{+}(A_2)
	=
	6.\]
\end{example}


We next show that the multiplicativity of the SNT-rank under the Kronecker product continues to hold for a broader class of symmetric nonnegative matrices under an additional assumption.

\begin{theorem}\label{theorem:kroneckersntequality_sntrank}
Let $A_{1} \in \mathcal{S}_{n_1}^{+}$ and $A_{2} \in \mathcal{S}_{n_2}^{+}$, where at least one of them has rank $1$. Then
\[
\operatorname{st}_{+}(A_1 \otimes A_2)
=
\operatorname{st}_{+}(A_1) \operatorname{st}_{+}(A_2).
\]
\end{theorem}

\begin{proof}
  Without loss of generality, assume that $\operatorname{rk}(A_1)=1$. Then there exists a nonzero vector $u = (u_i) \in \mathbb{R}_+^{n_1}$ such that $A_1 = uu^T$, and
\[
A_1 \otimes A_2 = \begin{pmatrix}
u_1^2 A_2 & u_1 u_2 A_2 & \cdots & u_1 u_{n_1} A_2\\
u_2 u_1 A_2 & u_2^2 A_2 & \cdots & u_2 u_{n_1} A_2\\
\vdots & \vdots & \ddots & \vdots \\
u_{n_1} u_1 A_2 & u_{n_1} u_2 A_2 & \cdots & u_{n_1}^2 A_2
\end{pmatrix}.
\]
Since $u \neq 0$, there exists $j \in \{1,\dots,n_1\}$ such that $u_j > 0$ and $u_{j}^{2}A_2$ is a principal submatrix of $ A_1 \otimes A_2$. By Proposition~\ref{prel:proposition2.2}(b),
\[
\operatorname{st}_{+}(A_2) \leq \operatorname{st}_{+}(A_1 \otimes A_2).
\]
Since $A_1$ is rank one matrix, by Proposition~\ref{prop:rank-snt-small}, $\operatorname{st}_{+}(A_{1})=1$. Using Theorem~\ref{lemma:kroneckerproduct}, we conclude that
\[
\operatorname{st}_{+}(A_1 \otimes A_2)
=
\operatorname{st}_{+}(A_1)\operatorname{st}_{+}(A_2).
\]
\end{proof}

\begin{remark}
A closer inspection of the proof of Theorem~\ref{theorem:kroneckersntequality_sntrank} reveals the following reformulation: Let $A_{1} \in \mathcal{S}_{n_1}^{+}$ and $A_{2} \in \mathcal{S}_{n_2}^{+}$,  where at least one of them has rank $1$. Then, the SNT-rank of  the Kronecker product satisfies
\[
\operatorname{st}_{+}(A_1 \otimes A_2)
=
\max\{\operatorname{st}_{+}(A_1), \operatorname{st}_{+}(A_2)\}.
\] 
\end{remark}

We next show that the SNT-rank is submultiplicative with respect to the Hadamard (Schur) product.
The Hadamard (Schur) product of two matrices $A=(a_{ij})$ and $B=(b_{ij})$ of same size is denoted by $A\circ B$
and is defined by
\[
A\circ B = (a_{ij} b_{ij}).
\]

\begin{proposition}\label{proposition:hadamardproduct}
Let $A_1$ and $A_2$ be two symmetric nonnegative matrices of order $n$. Then,
\[
\operatorname{st}_{+}(A_1 \circ A_2) \leq \operatorname{st}_{+}(A_1) \operatorname{st}_{+}(A_2).
\]
\end{proposition}

\begin{proof}
Let $A_1, A_2 \in \mathcal{S}_n^{+}$. Then $A_1 \circ A_2 \in \mathcal{S}_n^{+}$.
Since $A_1 \circ A_2$ is a principal submatrix of $A_1 \otimes A_2$,
the result follows from Proposition~\ref{prel:proposition2.2} and Theorem~\ref{lemma:kroneckerproduct}.
\end{proof}

The following example shows that, for any $n \ge 2$, the inequality in
Proposition~\ref{proposition:hadamardproduct} can be strict.

\begin{example}
For $n \ge 2$, let $A_1, A_2 \in \mathcal{S}_n^{+}$ be the matrices whose
leading $2 \times 2$ principal blocks are
\[
\begin{pmatrix}
0 & 1\\[2pt]
1 & 0
\end{pmatrix}
\quad \text{and} \quad
\begin{pmatrix}
1 & 0\\[2pt]
0 & 1
\end{pmatrix},
\]
respectively, and whose remaining entries are zero. 
Then $\operatorname{st}_{+}(A_1) = 2$ and $\operatorname{st}_{+}(A_2) = 2$, while $\operatorname{st}_{+}(A_1 \circ A_2) = 0$.
\end{example}
\subsection{\textbf{Kronecker Product and the nonnegative rank}}
In this subsection, we prove Theorems~\ref{theorem:kroneckersntequality_nonnegativerank} and \ref{theorem:kroneckersntequality_nonnegativerank2}, establishing that the nonnegative rank is multiplicative with respect to the Kronecker product under suitable assumptions. We begin with proving Theorem~\ref{theorem:kroneckersntequality_nonnegativerank}.

\begin{proof}[Proof of Theorem~\ref{theorem:kroneckersntequality_nonnegativerank}]
Using a standard factorization argument and the basic properties of the Kronecker product, one can derive the following inequality for all nonnegative matrices:
\begin{equation} \label{eqnkronrank1}
\operatorname{rk}_{+}(X_1 \otimes X_2)
\le
\operatorname{rk}_{+}(X_1)\operatorname{rk}_{+}(X_2) \hbox { for all } X_1 \in \mathbb{R}_{+}^{m_1 \times n_1} \hbox{ and } X_2 \in \mathbb{R}_{+}^{m_2 \times n_2}.
\end{equation}
Now, without loss of generality, assume that $\operatorname{rk}(A_1)=1$. Then, there exist nonzero vectors $u = (u_i) \in \mathbb{R}_+^{m_1}$ and $v = (v_i) \in \mathbb{R}_+^{n_1}$  such that $A_1 = uv^T$. Thus $\operatorname{rk}_{+}( A_1)=1$.  Since the nonnegative rank of a submatrix of a nonnegative matrix does not exceed that of the original matrix, by a similar argument as in the proof of Theorem \ref{theorem:kroneckersntequality_sntrank}, we have
\begin{equation}\label{eqnkronrank2}
\operatorname{rk}_{+}( A_2)
\le
\operatorname{rk}_{+}(A_1\otimes A_2).
\end{equation}
Combining \eqref{eqnkronrank1} and \eqref{eqnkronrank2} with $\operatorname{rk}_{+}( A_1)=1$ completes the proof.
\end{proof}

We next prove Theorems~\ref{theorem:kroneckersntequality_nonnegativerank2}. The proof requires the following basic result.

\begin{lemma}\cite[Corollary 4.2]{cohen1993nonnegative}\label{lemma:rk=rk+<=3}
$A \in \mathbb{R}_{+}^{m \times n}$ with either $m\leq 3$ or $n\leq 3$.  Then,
\[
\operatorname{rk}_{+}(A) 
= 
\operatorname{rk}(A).
\]
\end{lemma}

\begin{proof}[Proof of Theorem~\ref{theorem:kroneckersntequality_nonnegativerank2}]
From \eqref{eqnkronrank1}, \[\operatorname{rk}_{+}(A_1 \otimes A_2)
\le
\operatorname{rk}_{+}(A_1)\operatorname{rk}_{+}(A_2).\]  It remains to prove the reverse inequality.
By hypothesis either $m_i \le 3$ or $n_i \le 3$ for each $i=1,2$. By Lemma~\ref{lemma:rk=rk+<=3}, $\operatorname{rk}_{+}(A_1) = \operatorname{rk}(A_1)$ and $\operatorname{rk}_{+}(A_2) = \operatorname{rk}(A_2)$. Now, using the fact that $
\operatorname{rk}(X) \le \operatorname{rk}_{+}(X)
$ for any nonnegative matrix $X$, we obtain
\[
\operatorname{rk}_{+}(A_1)\operatorname{rk}_{+}(A_2)=\operatorname{rk}(A_1)\operatorname{rk}(A_2)=\operatorname{rk}(A_1 \otimes A_2)\leq \operatorname{rk}_{+}(A_1 \otimes A_2).
\]
This concludes the proof.
\end{proof}

\section{\textbf{SNT-rank and Euclidean distance matrices}}

In this section, we investigate the SNT-rank of  Euclidean distance matrices and prove Theorem~\ref{theorem:SNT_rank_Naturalno}. For $n\geq 3$, every Euclidean distance matrix has rank three \cite{mmlin2010edm}. Since Proposition~\ref{prop:rank-snt-small} shows that the SNT-rank of a symmetric nonnegative matrix coincides with its rank whenever the rank is at most two, rank-three matrices constitute the first nontrivial case. Indeed, the SNT-rank of a rank-three symmetric nonnegative matrix may strictly exceed its rank (see Example~\ref{example:LEDM6_st=5}). More generally, Shitov \cite[Section~3]{shitov2025further} proved that every rank-three symmetric nonnegative matrix $A$ of order $n$ satisfies $
\operatorname{st}_{+}(A)\leq 294n^{2/3}$, and, in the special case of the Euclidean distance matrix $A(1,\ldots,n)$, obtained the sharper bound $\operatorname{st}_{+}\!\big(A(1,\ldots,n)\big)
\leq
4\log_2 n +4.$ Our goal in this section is to obtain improved bounds for the SNT-rank of several classes of  Euclidean distance matrices. We begin with the following definition.

\begin{definition}
	 For an index set $\alpha\subseteq\{1,\ldots,n\},$ let \(\bar{\alpha}\) denote its complement in \(\{1,\ldots,n\}\). If $A \in \mathbb{R}^{n \times n}$ and $\alpha,\beta\subseteq\{1,\ldots,n\},$ then \(A_{\alpha\beta}\) denotes the submatrix of \(A\) formed by the rows indexed by \(\alpha\) and the columns indexed by \(\beta\). If \(\alpha=\beta\), the principal submatrix \(A_{\alpha \alpha}\) is abbreviated to $A_{\alpha}$.
\end{definition}
 We now prove our first main result Theorem \ref{theorem:SNT_rank_Naturalno}.

\begin{proof}[Proof of Theorem \ref{theorem:SNT_rank_Naturalno}]
By Proposition~\ref{proposition:principal_rearrangement/diagonal_scaling}$(a)$, the SNT-rank of a symmetric nonnegative matrix is invariant under any principal rearrangement,
\begin{equation}\label{eqaution:SNT-rankqual}
\operatorname{st}_{+}\big(A(1, \ldots, n, n+1, \ldots, 2n)\big)=\operatorname{st}_{+}\big(A(1, \ldots, n, 2n, \ldots, n+1)\big).  
\end{equation}
Let $\alpha = \{1, \ldots, n\}, ~\bar{\alpha} = \{2n, \ldots, n+1\}$ and
set $M :=A(1, \ldots, n, 2n, \ldots, n+1)$. Then 
\[
M =
\begin{pmatrix}
M_{\alpha} & M_{\alpha\bar{\alpha}} \\[4pt]
M_{\bar{\alpha}\alpha} & M_{\bar{\alpha}}
\end{pmatrix}
\in \mathcal{S}_{2n}^{+}
\]
and $M_{\alpha} = M_{\bar{\alpha}}= A(1, \ldots, n)$. Next, consider the off-diagonal block $M_{\alpha\bar{\alpha}}$. For $i \in \alpha$ and $j \in \bar{\alpha}$,
\[
(M_{\alpha\bar{\alpha}})_{ij} 
= \big(i - (2n - j + 1)\big)^2 
= (i - j)^2 + (2n - 2i + 1)(2n - 2j + 1).
\]
Define $X:= (x_{ij})$ with $x_{ij} = (2n - 2i + 1)(2n - 2j + 1)$. Then
\[
M_{\alpha\bar{\alpha}} = M_{\alpha} + X,
\]
and by the symmetry of $M$, $M_{\bar{\alpha}\alpha}= M_{\alpha\bar{\alpha}}^{T} =  M_{\alpha} + X$. Thus,
\begin{eqnarray}
A(1, \ldots, n, 2n, \ldots, n+1)
&=&
\begin{pmatrix}
M_{\alpha} & M_{\alpha\alpha} + X \\[4pt]
M_{\alpha\alpha} + X & M_{\alpha}
\end{pmatrix} \nonumber \\
&=&
\begin{pmatrix}
1 & 1\\
1 & 1
\end{pmatrix}
\otimes M_{\alpha}
+
\begin{pmatrix}
0 & 1\\
1 & 0
\end{pmatrix}
\otimes X.\label{equation:JS-matrix kroneckerproduct}
\end{eqnarray}
Since $\operatorname{st}_{+}
\!\left(
\begin{pmatrix}
1 & 1\\
1 & 1
\end{pmatrix}
\right)
=1,$ $\operatorname{st}_{+}
\!\left(
\begin{pmatrix}
0 & 1\\
1 & 0
\end{pmatrix}
\right)
=2,$ applying Proposition~\ref{prel:proposition2.2} and Theorem~\ref{lemma:kroneckerproduct} to \eqref{equation:JS-matrix kroneckerproduct}, together with the fact that $X$ has SNT-rank one, yields
\[
\operatorname{st}_{+}\!\big(A(1, \ldots, n, 2n, \ldots, n+1)\big)
\leq
\operatorname{st}_{+}\!\big(A(1, \ldots, n)\big)+2.
\]
Now, using \eqref{eqaution:SNT-rankqual}, we obtain the following relation for the SNT-rank:
\begin{equation}\label{equation:SNT-rank_main equation}
\operatorname{st}_{+}\!\big(A(1, \ldots, n, n+1, \ldots, 2n)\big)
\leq \operatorname{st}_{+}\!\big(A(1, \ldots, n)\big) + 2.   
\end{equation}
With this inequality in hand, we now complete the proof by considering the following two cases.
\smallskip

\textbf{Case 1:} 
Note that $\operatorname{st}_{+}(A(1)) = 0$. If $n > 1$ and $n$ is a power of two, then using the recurrence relation in \eqref{equation:SNT-rank_main equation}, we obtain
\[
\operatorname{st}_{+}\!\big(A(1, \ldots, n)\big)
\leq 2 \log_2 n.
\]

\textbf{Case 2:}
If $n$ is not a power of two, then $n < 2^{\lceil \log_2 n \rceil}$, and hence $A(1, \ldots, n)$ is a principal submatrix of $A(1, \ldots, n, \ldots, 2^{\lceil \log_2 n \rceil})$. Therefore, by Proposition~\ref{prel:proposition2.2} and the result of Case~1, we have
\[
\operatorname{st}_{+}\!\big(A(1, \ldots, n)\big) \leq \operatorname{st}_{+}\!\big(A(1, \ldots, n, \ldots, 2^{\lceil \log_2 n \rceil})\big) \leq 2 \lceil \log_2 n \rceil.
\]
Combining the two cases, we conclude that
\[
\operatorname{st}_{+}\!\big(A(1, \ldots, n)\big)
\leq 2 \lceil \log_2 n \rceil.
\]
This completes the proof.
\end{proof}


The following table summarizes several exact values and sharp bounds for the SNT-rank of  Euclidean distance matrices. The values and bounds are obtained by combining the lower bounds for the nonnegative rank from Table~1 of \cite{gillis2012geometric} and Table~1 of \cite{vandaele2016heuristics} with the upper bounds established in  Theorem~\ref{theorem:SNT_rank_Naturalno} and \eqref{helbnd1}.
\begin{table}[H]
\centering
\tiny
\caption{SNT-ranks and bounds for the matrices $A(1,\ldots,n)$.}
\label{tab:LEDM_snt}
\renewcommand{\arraystretch}{1.12}
\begin{tabular}{@{}ccccl@{}}
\toprule
$n$
& $\operatorname{rk}(A(1,\ldots,n))$
& $\operatorname{rk}_{+}(A(1,\ldots,n))$
& $\operatorname{st}_{+}(A(1,\ldots,n))$
& Status \\
\midrule
1  & 0 & 0  & 0 & Trivial case \\
2  & 2 & 2  & 2 & Exact value \\
3  & 3 & 3  & 3 & Exact value \\
4  & 3 & 4  & 4 & Exact value \\
5  & 3 & 5  & 5 & Exact value \\
6  & 3 & 5  & 5 & Exact value \\
7  & 3 & 6  & 6 & Exact value \\
8  & 3 & 6  & 6 & Exact value \\
10 & 3 & 7  & 7 & Exact value \\
12 & 3 & 7  & $7\text{--}8$ & Partial bounds \\
16 & 3 & 8  & 8 & Exact value \\
32 & 3 & $9\text{--}10$ & $9\text{--}10$ & Partial bounds \\
\bottomrule
\end{tabular}
\end{table}

Table~\ref{tab:LEDM_snt} illustrates that, although the usual rank of $A(1,\ldots,n)$ remains equal to 3 for all $n\geq3$, its SNT-rank increases with $n$.

\medskip

As immediate consequences of  Theorem~\ref{theorem:SNT_rank_Naturalno}, we now provide upper bounds for the SNT-ranks of several linear Euclidean distance matrices.

\begin{corollary}\label{cor:lEDM_integer_snt}
Let $x_1 \le x_2 \le \cdots \le x_n$ be integers, and let
$A(x_1,\ldots,x_n)\in\mathcal{S}_n^{+}$ be the matrix defined in
\eqref{eqnedm}.
Then
\[
\operatorname{st}_{+}\!\big(A(x_1,\ldots,x_n)\big)
\leq
2\left\lceil
\log_2(x_n-x_1+1)
\right\rceil.
\]
\end{corollary}

\begin{proof}
Observe that $A(x_1,\ldots,x_n)$ is invariant under translating all its parameters by the same integer, that is,
\[
A(x_1,\ldots,x_n)
=
A(x_1+k,\ldots,x_n+k),
\qquad
\text{for every integer } k
\]
Set $m:=x_n-x_1+1.$ Translating the parameters by $1-x_1$, we may assume without loss of generality that
 $x_1,\ldots,x_n\in\{1,2,\ldots,m\}.$ Hence, $A(x_1,\ldots,x_n)$ is a principal submatrix of $A(1,\ldots,m)$. By Proposition~\ref{prel:proposition2.2}(b) and Theorem~\ref{theorem:SNT_rank_Naturalno},
\[
\operatorname{st}_{+}\!\big(A(x_1,\ldots,x_n)\big)
\leq
\operatorname{st}_{+}\!\big(A(1,\ldots,m)\big)
\leq
2\lceil \log_2 m\rceil.
\]
Finally, substituting $m=x_n-x_1+1$ into the above inequality completes the proof.
\end{proof}

\smallskip
We next derive an upper bound for the SNT-rank of the Euclidean distance matrix associated with the harmonic sequence  $x_i=\frac1i,$  for $i=1,\ldots,n.$

\begin{corollary}
$\operatorname{st}_{+}\!\big(A(1,1/2,\ldots,1/n)\big)
\leq
2\lceil\log_2 n\rceil.$
\end{corollary}

\begin{proof}
For $i,j=1,\ldots,n$,
\[
(A(1,1/2,\ldots,1/n))_{ij}
=
\left(\frac1i-\frac1j\right)^2
=
\frac{(i-j)^2}{i^2j^2}.
\]
Define $D:
=
\operatorname{diag}
\left(
1,\frac1{2^2},\ldots,\frac1 {n^2}
\right)$. Then
\[
A(1,1/2,\ldots,1/n)
=
D\,A(1,\ldots,n)\,D.
\]
Using Proposition~\ref{proposition:principal_rearrangement/diagonal_scaling}(b) and Theorem~\ref{theorem:SNT_rank_Naturalno}, we have
\[
\operatorname{st}_{+}\!\big(A(1,1/2,\ldots,1/n)\big)=\operatorname{st}_{+}\!\big(A(1,2,\ldots,n)\big)
\leq
2\lceil\log_2 n\rceil.\]
\end{proof}

We now extend the preceding results to Euclidean distance matrices associated with rational parameters.

\begin{corollary}\label{cor:rational_ledm}
Let $x_i=\frac{p_i}{q_i},$ where $p_i\in\mathbb Z$ and $q_i\in\mathbb N$ for $i =1, \ldots,n$. Then
\[
\operatorname{st}_{+}\!\big(A(x_1,\ldots,x_n)\big)
\le
2\left\lceil
\log_2\bigl(\max_i y_i-\min_i y_i+1\bigr)
\right\rceil,
\]
where $y_i=\operatorname{lcm}(q_1,\ldots,q_n)x_i$ for $i=1,\ldots,n.$
\end{corollary}

\begin{proof}
Set $Q:=\operatorname{lcm}(q_1,\ldots,q_n)$. Then $y_i:=Qx_i$ is an integer for each $i=1,\ldots,n$. Since
\[
(y_i-y_j)^2
=
Q^2(x_i-x_j)^2,
\qquad i,j=1,\ldots,n,
\]
and $Q^2>0$, it follows that
\[
\operatorname{st}_{+}\!\big(A(x_1,\ldots,x_n)\big)
=
\operatorname{st}_{+}\!\big(A(y_1,\ldots,y_n)\big).
\]
By Proposition~\ref{proposition:principal_rearrangement/diagonal_scaling}, the SNT-rank is invariant under simultaneous permutations of rows and columns. Therefore, without loss of generality, we may assume that
\[
y_1\le\cdots\le y_n.
\]
Now, applying Corollary~\ref{cor:lEDM_integer_snt} to the integers
$y_1,\ldots,y_n$, we obtain
\[
\operatorname{st}_{+}\!\big(A(y_1,\ldots,y_n)\big)
\le
2\left\lceil
\log_2(y_n-y_1+1)
\right\rceil.
\]
Since $y_1\le\cdots\le y_n$, the desired conclusion follows immediately
\end{proof}

\medskip
We now move beyond rational parameters and study Euclidean distance matrices associated with real numbers. To this end, we first establish the following lemma, which provides a class of symmetric nonnegative matrices with SNT-rank at most four. This lemma will play a key role in the proof of the next theorem.

\begin{lemma}\label{lemma:XX^{T}}
Let $x_1,\dots,x_n$ and $y_1,\dots,y_n$ be real numbers satisfying
\[
\min_{1\le p\le n}|x_p|
\ge
\max_{1\le q\le n} y_q.
\]
Define $X := (|x_i|-y_j)\in \mathbb{R}^{n \times n}$ and  $A:=
\begin{pmatrix}
0 & X\\
X^{T} & 0
\end{pmatrix}$. Then $A\in\mathcal S_{2n}^{+}$ and $\operatorname{st}_{+}(A)\le4.$
\end{lemma}

\begin{proof}
Since $\min_{1\le p\le n}|x_p|
\ge
\max_{1\le q\le n} y_q,$ it follows that $x_{ij}:=|x_i|-y_j\ge0$ for $i,j=1,\ldots,n.$
Thus $A\in\mathcal S_{2n}^{+}.$ To show $\operatorname{st}_{+}(A)\le4$, set $b:=\max_{1\le q\le n} y_q,$ $u_i:=|x_i|-b,$ and $v_j:=b-y_j$ for all $i,j$.
Then $u_i\geq0$ and  $v_j\ge0$ for all $i,j$, and $X
=
u\mathbf 1^{T}
+
\mathbf 1 v^{T},$
where $\mathbf 1$ denotes the all-ones vector. Thus, $\operatorname{rk}_{+}(X)\le2.$ Also note that $\operatorname{rk}_{+}\left(\begin{pmatrix}
	0 & X\\
	0 & 0
\end{pmatrix}\right)=\operatorname{rk}_{+}(X)$. Therefore, by \cite[Proposition 2.3]{bukovvsek2023symmetric}, we have
\[
\operatorname{st}_{+}(A)\leq
2\operatorname{rk}_{+}(X)
\le4.\]
\end{proof}
Using the preceding lemma, we now derive the following recurrence relation.
\begin{theorem}
Let $x_1, \ldots, x_n$ be real numbers, and let $r > 0$ satisfy
\begin{equation}\label{eq:aboveRationals}
\min_{1 \le p \le n} \, r x_p^{2} \;\ge\; \max_{1 \le q \le n} x_q^{2}.  
\end{equation}

Then
\[
\operatorname{st}_{+}\!\big(A(x_1, \ldots, x_n, r x_1, \ldots, r x_n)\big)
\leq 
\operatorname{st}_{+}\!\big(A(x_1, \ldots, x_n)\big) + 4.
\]
\end{theorem}

\begin{proof}
If $x_i=0$ for some $i \in \{1, \ldots,n\}$, then by \ref{eq:aboveRationals} all $x_i=0$ and the assertion follows trivially. Let $x_i \neq 0$ for all $i$.
Let $\alpha = \{1, \ldots, n\}$ and $\bar{\alpha} = \{n+1, \ldots, 2n\}$ and set $A:=A(x_1, \ldots, x_n, r x_1, \ldots, r x_n)$. Then
\[
A(x_1, \ldots, x_n, r x_1, \ldots, r x_n) =
\begin{pmatrix}
A_{\alpha} & A_{\alpha\bar{\alpha}} \\[4pt]
A_{\bar{\alpha}\alpha} & A_{\bar{\alpha}}
\end{pmatrix}
\in \mathcal{S}_{2n}^{+}.
\]
Note that $A_{\alpha}= A(x_1, \ldots, x_n)$ and $A_{\bar{\alpha}}= r^{2}A_{\alpha}$. We next consider the off-diagonal block $A_{\alpha\bar{\alpha}}$. For $i \in \alpha$ and $j \in \bar{\alpha}$,
\[
(A_{\alpha\bar{\alpha}})_{ij} = (x_i - r x_j)^2 = r (x_i - x_j)^{2} + (r-1)\,(r x_j^{2} - x_i^{2}).
\]
Define $X:=(r x_j^{2} - x_i^{2}) \in \mathbb{R}^{n \times n}$. Then
\[
A_{\alpha\bar{\alpha}}
= r A_{\alpha} + (r-1) X.
\]
By hypothesis, $\min_{1 \le p \le n} r x_p^{2} \;\ge\; \max_{1 \le q \le n} x_q^{2}$. Thus, $X$ is a nonnegative matrix and $r\geq 1$. Since $A$ is symmetric,

\begin{align}\label{equation:JS-matrix_kroneckerproduct}
A(x_1, \ldots, x_n, r x_1, \ldots, r x_n)
&=
\begin{pmatrix}
A_{\alpha} & rA_{\alpha} + (r-1)X \\[4pt]
rA_{\alpha} + (r-1)X^{T} & r^{2}A_{\alpha}
\end{pmatrix} \\[3pt]
&=
\begin{pmatrix}
1 & r\\
r & r^2
\end{pmatrix}
\otimes A_{\alpha}
+
(r-1)\,\begin{pmatrix}
0 & X\\
X^{T} & 0
\end{pmatrix} \nonumber.
\end{align}
Note that $\operatorname{st}_{+}
\!\left(
\begin{pmatrix}
1 & r\\
r & r^2
\end{pmatrix}
\right)
=1$ and by Lemma~\ref{lemma:XX^{T}}, $\operatorname{st}_{+}
\!\left(
\begin{pmatrix}
0 & X\\
X^{T} & 0
\end{pmatrix}
\right)
\leq 4.$ Applying Proposition~\ref{prel:proposition2.2} and Theorem~\ref{lemma:kroneckerproduct} to \eqref{equation:JS-matrix_kroneckerproduct}, we have
\[
\operatorname{st}_{+}\!\big(A(x_1, \ldots, x_n, r x_1, \ldots, r x_n)\big)
\leq 
\operatorname{st}_{+}\!\big(A(x_1, \ldots, x_n)\big) + 4.
\]
\end{proof}
We conclude this section by establishing a lower bound on the SNT-rank of Euclidean distance matrices with pairwise distinct parameters. In fact, we prove a more general result by establishing a lower bound for the SNT-rank of all symmetric nonnegative matrices with zero diagonal and positive off-diagonal entries. 

\begin{proposition}\label{proposition:snt_rk=4}
	Let $A \in \mathcal{S}_n^{+}$ be a matrix with zero diagonal and positive off-diagonal entries. Then
	\[
	\operatorname{st}_{+}(A) 
	\geq 
	\min \left\{ k : n \leq \binom{k}{\lfloor k/2 \rfloor} \right\}.
	\]
	In particular, every matrix $A \in \mathcal{S}_4^{+}$ with zero diagonal and positive off-diagonal entries satisfies
	\[
	\operatorname{rk}_{+}(A)=\operatorname{st}_{+}(A)=4.
	\]
\end{proposition}

\begin{proof}
	Let $A \in \mathcal{S}_n^{+}$ be a symmetric nonnegative matrix with zero diagonal and positive off-diagonal entries. By \cite[Proposition 1.3]{beasley2009realrank},
	\[
	\operatorname{rk}_{+}(A)
	\geq
	\min \left\{
	k : n \leq \binom{k}{\lfloor k/2 \rfloor}
	\right\}.
	\]
	Since $\operatorname{rk}_{+}(A)\leq \operatorname{st}_{+}(A),$ it follows that
	\[
	\operatorname{st}_{+}(A)
	\geq
	\min \left\{
	k : n \leq \binom{k}{\lfloor k/2 \rfloor}
	\right\}.
	\]
	\medskip
	
	To prove the second assertion, assume that $A \in \mathcal{S}_4^{+}$ with zero diagonal and positive off-diagonal entries. Since $\min \left\{
	k : 4 \leq \binom{k}{\lfloor k/2 \rfloor}
	\right\}=4,$ $\operatorname{st}_{+}(A)\geq \operatorname{rk}_{+}(A) \geq  4. $
	By Proposition~\ref{prel:proposition2.1}, we have $\operatorname{rk}_{+}(A)=\operatorname{st}_{+}(A)=4.$
\end{proof}

\begin{example}\label{ex:rk=3<st=4}
	Let  $	B=
	\begin{pmatrix}
		0 & 1 & 1 & b\\
		1 & 0 & \frac{4}{b} & 1\\
		1 & \frac{4}{b} & 0 & 1\\
		b & 1 & 1 & 0
	\end{pmatrix},$ where
	 $b>0$.  By Proposition~\ref{proposition:snt_rk=4}, $\operatorname{st}_{+}(B)=4$. Moreover, $ \operatorname{rk}(B)=3$. Therefore, the matrix $A= B \oplus I_{n-4}\in \mathbb{R}^{n \times n}$ provides another example demonstrating that the equality in Corollary~\ref{cor:rk=snt<=3} does not hold in general.
\end{example}
\vspace*{-0.2cm}
\section*{\textbf{Acknowledgement}}
P.N. Choudhury thanks Helena Šmigoc for introducing him to this topic. B.P. Chauhan was supported by the IIT Gandhinagar Post-Doctoral Fellowship IP/IP/50020 (Indian Institute of Technology Gandhinagar, India). P.N. Choudhury\ was partially supported by
ANRF Prime Minister Early Career Research Grant ANRF/ECRG/2024/002674/PMS (ANRF, Govt.~of India), INSPIRE Faculty Fellowship Research Grant DST/INSPIRE/04/2021/002620 (DST, Govt.~of India).


\begin{thebibliography}{99}

\bibitem{arora2013practicalalgorithm} S. Arora, R. Ge, Y. Halpern, D. Mimno, A. Moitra, D. Sontag, Y. Wu, and M. Zhu, 
A practical algorithm for topic modeling with provable guarantees, \emph{Proceedings of the 30th International Conference on Machine Learning}, 
28:280--288, 2013.


\bibitem{beasley2009realrank} L. B. Beasley and T. J. Laffey, 
Real rank versus nonnegative rank, 
\emph{Linear Algebra Appl.}, 
431(12):2330--2335, 2009.

\bibitem{Dagstuhl report 13082} L. Beasley, T. Lee, H. Klauck and D. O. Theis, Dagstuhl report 13082: communication complexity, linear
optimization, and lower bounds for the nonnegative rank of matrices, \emph{arXiv:1305.4147}, 
2013.

\bibitem{berman2003completely} A. Berman and N. Shaked-Monderer, 
Completely positive matrices, 
\emph{World Scientific}, 
2003.

\bibitem{bukovvsek2023symmetric} D. K. Bukovšek and H. Šmigoc, 
Symmetric nonnegative matrix trifactorization, 
\emph{Linear Algebra Appl.}, 
665:36--60, 2023.

\bibitem{bukovvsek2025pattern} D. K. Bukovšek and H. Šmigoc, 
Symmetric nonnegative trifactorization of pattern matrices,
\emph{Linear Algebra Appl.},
721, 310--338, 2025.


\bibitem{cohen1993nonnegative} J. E. Cohen and U. G. Rothblum, 
Nonnegative ranks, decompositions, and factorizations of nonnegative matrices, 
\emph{Linear Algebra Appl.}, 
190:149--168, 1993.



\bibitem{fu2018anchor} X. Fu, K. Huang, N.D. Sidiropoulos, Q. Shi, and M. Hong, Anchor-free correlated topic modeling, 
\emph{IEEE Trans. Pattern Anal. Mach. Intell.}, 
41(5):1056--1071, 2019.


\bibitem{gillis2012geometric} N. Gillis and F. Glineur, 
On the geometric interpretation of the nonnegative rank, 
\emph{Linear Algebra Appl.}, 
437(11):2685--2712, 2012.

\bibitem{gillis2020siam} N. Gillis, 
Nonnegative matrix factorization, 
\emph{Society for Industrial and Applied Mathematics}, 
2020.

\bibitem{ho2008thesis} N. D. Ho, 
Nonnegative matrix factorizationalgorithms and applications, 
PhD thesis, Catholic University of Louvain, Louvain-la-Neuve, Belgium, 
2008.


\bibitem{lee1999learning} D. D. Lee and H. S. Seung, 
Learning the parts of objects by non-negative matrix factorization, 
\emph{Nature}, 
401:788--791, 1999.

\bibitem{mmlin2010edm} M. M. Lin, M. T. Chu, 
On the nonnegative rank of Euclidean distance matrices, 
\emph{Linear Algebra Appl.}, 
433(3):681--689, 2010.

\bibitem{shitov2025further} Y. Shitov,
Further restricted versions of the nonnegative matrix factorization problem,
\emph{Linear Algebra Appl.}, 
710:267--272, 2025.

\bibitem{vandaele2016heuristics} A. Vandaele, N. Gillis, F. Glineur, and D. Tuyttens, 
Heuristics for exact nonnegative matrix factorization, 
\emph{J. Global Optim.},
65(2):369--400, 2016.

\end{thebibliography}
\end{document}